\documentclass[11pt,reqno]{amsart}

\usepackage[T1]{fontenc}
\usepackage{amsmath}
\usepackage{amssymb}
\usepackage{mathtools}
\usepackage[margin=1in]{geometry}
\usepackage[colorlinks=true,citecolor=blue,linkcolor=blue,urlcolor=blue]{hyperref}
\allowdisplaybreaks

\newtheorem{theorem}{Theorem}
\newtheorem{definition}[theorem]{Definition}
\newtheorem{proposition}[theorem]{Proposition}
\newtheorem{lemma}[theorem]{Lemma}
\newtheorem{corollary}[theorem]{Corollary}

\newtheorem{remark}[theorem]{Remark}

\title[Sharp CMLSI on tori]
{Sharp Complete Modified Log-Sobolev Inequalities on Classical and
Quantum Tori}

\author{Long Zhao}
\address{School of Mathematics and Statistics, Wuhan University, Wuhan 430072, China}
\email{zhaolong@whu.edu.cn}
\date{\today}


\date{\today}

\begin{document}

\begin{abstract}
We prove that the heat semigroup on the circle has optimal complete
modified logarithmic Sobolev constant \(1\).  The proof is based on a 
matrix-valued Wirtinger inequality and yields the stronger
Bogoliubov--Kubo--Mori Fisher information contraction with rate
\(e^{-2t}\).  As a consequence, the complete tensorization and transference principle yields the same sharp constant for the heat semigroups on classical and quantum tori.
\end{abstract}

\maketitle

\section{Introduction}
\label{sec:introduction}

Since Gross's seminal work relating logarithmic Sobolev inequalities (LSI) with
hypercontractivity \cite{Gross1975LSI}, these inequalities
have been extensively studied on manifolds and graphs because of
their connections with geometry and concentration of measure\cite{Ledoux2001Concentration,DiaconisSaloffCoste1996,
BakryGentilLedoux2014}. For the unit circle $\mathbb T=\mathbb R/(2\pi\mathbb Z)$, Weissler in \cite{Weissler1980Circle} proved that for any smooth function $f:\mathbb{T}\to \mathbb{R}$
\begin{equation}
  \int f^2\log f^2\,d\mu-(\int f^2d\mu) \log(\int f^2d\mu)
 \le 2\int|f'|^2d\mu,
 \label{eq:LSI}
\end{equation}
where $\mu$ is the uniform measure.
See \cite{Emery87} for an elementary proof was later given by
\'{E}mery and Yukich. Taking $g=f^2$, \eqref{eq:LSI} admits an equivalent formulation, called modified log-Sobolev inequality (MLSI)
\begin{equation}
  \int g\log g\,d\mu-(\int g\,d\mu)\log (\int g\,d\mu)
 \le -\frac12\int\!
 g{''}\log g
 d\mu.
 \label{eq:MLSI}
\end{equation}
In this work, we prove the following MLSI for matrix-valued functions.
\begin{theorem}\label{thm:sharp-circle}
For every \(n\geq1\)  and for every smooth, pointwise positive, matrix-valued function
\(g:\mathbb T\to M_n^+\),
\begin{equation}
 \operatorname{tr}\!\left(
  \int_{\mathbb T}g\log g\,d\mu-(\mathbb Eg)\log(\mathbb Eg)
 \right)
 \le \frac12\operatorname{tr}\left(\int_{\mathbb T}\!
  (\Delta g)(\log g-\log\mathbb Eg)
 d\mu\right),
 \label{eq:main}
\end{equation}
where $\mathbb Eg=\int_{\mathbb T}g\,d\mu$ is the matrix-valued mean, $\Delta g= -g''$ is the Laplacian and $\operatorname{tr}$ denotes the matrix trace.
\end{theorem}
Note that the above inequality is sharp, as the scalar-valued case $n=1$ is the equivalent form \eqref{eq:MLSI} of Weissler's inequality is already known to be sharp.
Theorem 1 shows that the MLSI holds \eqref{eq:MLSI} uniformly for matrix-valued functions independent of the matrix dimension $n\ge 1$.

The above uniform matrix-valued MLSI, called complete modified Log-Sobolev inequality (CMLSI), was introduced in \cite{GaoJungeLaRacuente2020Fisher} for the tensorization of quantum Markov semigroups.  Theorem \ref{thm:sharp-circle} basically says that for the heat semigroup $H_t=e^{-\Delta t}$ on the unit circle,
\[\alpha_{c}(\Delta)=1 \]
where \(\alpha_{\mathrm c}\) denotes the optimal CMLSI
constant. The previously known lower bound obtained was strictly smaller than \(1\) \cite{BrannanGaoJunge2022}. 
The sharp constant $1$ was posed as an open question in
\cite[Problem~6.5]{GaoJungeLaRacuenteLi2025CPOrder}. We answer this question affirmatively. 

One application of the operator-valued inequality above is
transference to noncommutative spaces. For a real skew-symmetric
\(d\times d\) matrix \(\theta\), let \(C(\mathbb T_\theta^d)\) be the
universal unital \(C^*\)-algebra generated by unitaries
\(U_1,\ldots,U_d\) satisfying
\(U_jU_k=e^{2\pi i\theta_{jk}}U_kU_j\). For
\(m=(m_1,\ldots,m_d)\in\mathbb Z^d\), write
\(U^m=U_1^{m_1}\cdots U_d^{m_d}\). The canonical trace
\(\tau_\theta\) is determined by
\(\tau_\theta(U^m)=\delta_{m,0}\), and we denote by
\(\mathcal R_\theta^d
=\pi_{\tau_\theta}(C(\mathbb T_\theta^d))''\)
the corresponding von Neumann algebra. Its standard heat generator
is the Fourier multiplier
\[\Delta_{\theta}(U^m)=|m|^2U^m,\quad \text{ where }\quad |m|^2=m_1^2+\cdots+m_d^2,\] 
and its fixed-point expectation is
\(E_\theta(x)=\tau_\theta(x)1\). When \(\theta=0\), this reduces to the
classical heat generator \(\Delta_{0}\) on \(\mathbb T^d\). 
Complete tensorization \cite[Section~7.1]{GaoJungeLaRacuente2020Fisher} followed by the gauge transference principle \cite[Section~6.2]{BrannanGaoJunge2023II} yields the sharp CMLSI constant on both classical and quantum tori.
\begin{corollary}[Classical and quantum tori]
\label{cor:sharp-tori}
For every \(d\geq1\) and every real skew-symmetric \(\theta\),
\begin{equation}
    \alpha_{\mathrm c}(\Delta_{\theta})
    =\alpha(\Delta_{\theta})=1.
    \label{eq:quantum-torus-sharp-constant}
\end{equation}
\end{corollary}

Thus the sharp constant is independent of both the dimension and the
deformation parameter.  Positive CMLSI estimates were known previously,
whereas the sharp constant $1$ in \eqref{eq:quantum-torus-sharp-constant}, even
for the ordinary MLSI $\alpha(\Delta_{\theta})$ on a genuinely quantum torus, remained open. Our result here shows the entropy decay rate for amplified dynamics $H_t\otimes \text{id}_{M_n}$. To the best of the author's knowledge, the sharp LSI constant on quantum tori remains open.

In the following, we give the proof of Theorem \ref{thm:sharp-circle}. The tensorization and transference argument for Corollary \ref{cor:sharp-tori} is now standard, which we leave the details for the readers.\\

\noindent {\bf Acknowledgement.} 
LZ is partially supported by the National Natural Science Foundation
of China (grant No.~12401163) and the Department of Science and Technology
of Hubei Province (project Nos.~2025EHA041 and 2025AFA044).
The author is deeply grateful to my advisor Li Gao for his guidance,
many helpful discussions, and valuable advice on the writing of this manuscript.
The author acknowledges the use of AI tools during the exploratory stage of
this project.  All mathematical arguments and proofs in the final manuscript
were checked and written by the author.


\section{Proof of Theorem 1}
\label{sec:sharp-circle}

\subsection{Matrix-valued entropy and Fisher information}
\label{subsec:circle-setup}
Consider the unit circle equipped with normalized uniform measure
\[
    \mathbb T=\mathbb R/(2\pi\mathbb Z),
    \qquad
    d\mu(\theta)=\frac{d\theta}{2\pi}.
\]
We denote the standard Laplacian as
\[
    \Delta=-\partial_\theta^2,
    \qquad
    \Delta(z^k)=k^2z^k,   \qquad z(\theta)=e^{i\theta},
    \quad k\in\mathbb Z .
\]
Thus the heat semigroup \(H_t=e^{-t\Delta}\) satisfies
\[ H_t(z^k)=e^{-k^2t}z^k , k\in\mathbb Z.\]

Fix \(n\geq1\).  We work with matrix-valued functions
\[
    f:\mathbb T\longrightarrow M_n
\]
and let \(H_t\) act entrywise.  Throughout this section,
\(\operatorname{tr}\) denotes the standard matrix trace; its
normalization does not affect any constant in the discussion.  The 
matrix-valued expectation is
\[
    \mathbb{E}f=\int_{\mathbb T}f(\theta)\,d\mu(\theta)\in M_n,
\]
viewed also as a constant matrix-valued function on \(\mathbb T\).

For \(f\in L_\infty(\mathbb T;M_n)_+\) being pointwise positive, define the matrix-valued
relative entropy
\begin{equation}
    D(f\Vert \mathbb{E}f)
      =
      \int_{\mathbb T}\operatorname{tr}(f\log f)\,d\mu
       -\operatorname{tr}\bigl((\mathbb{E}f)\log(\mathbb{E}f)\bigr) =
      \int_{\mathbb T}
      \operatorname{tr}\bigl(
        f(\log f-\log \mathbb{E}f)
      \bigr)\,d\mu,
    \label{eq:circle-entropy}
\end{equation}
where \(0\log0=0\). This is exactly the classical-quantum relative entropy from the matrix-valued $f$ with respect to its mean $\mathbb{E}f$.

For a smooth, uniformly positive
\(g:\mathbb T\to M_n\), its Fisher information is
\begin{align}
    \mathcal I(g)
      =
      \int_{\mathbb T}
      \operatorname{tr}\bigl((\Delta g)
        (\log g-\log \mathbb{E}g)\bigr)\,d\mu
       =
      \int_{\mathbb T}
      \operatorname{tr}\bigl(-g''(\theta)\log g(\theta)\bigr)
      \,d\mu(\theta).
    \label{eq:circle-fisher}
\end{align}
The term involving \(\log \mathbb{E}g\) vanishes because
\(\mathbb{E}(\Delta g)=0\). 
\begin{definition}
The complete modified log Sobolev constant $\alpha_c(\Delta)$ of the semigroup $H_t=e^{-\Delta t}$ is then defined as the largest \(\alpha\) such that for all smooth positive matrix-valued  $g:\mathbb{T}\to M_n$ and all matrix sizes \(n\).
\begin{equation}
        2\alpha\,D(g\Vert \mathbb{E}g)\leq\mathcal I(g).
    \label{eq:circle-cmlsi}
\end{equation}
Namely, $$\alpha_c(\Delta):=\inf_{n\ge 1}\inf_{g\in C^\infty(\mathbb{T},M_n) } \frac{\mathcal I(g)}{ 2D(g\Vert \mathbb{E}g)} .$$
\end{definition}
Both matrix-valued entropy and Fisher information are homogeneous under \(f\mapsto cf\), \(c>0\), so it suffices to consider
the normalization to matrix-valued densities 
\(f\ge 0, \int_{\mathbb T}\operatorname{tr}(f)\,d\mu=1\). Note that the Fisher information is the derivative of entropy over the heat semigroup time evolution
\[ \mathcal I(g)=-\frac{d}{dt} D(H_t g \Vert \mathbb{E}g)|_{t=0}\ . \]
By standard argument of Gronwall's lemma, \eqref{eq:circle-cmlsi} is equivalent to 
\begin{equation}
    D(H_tf\Vert \mathbb{E}f)
      \leq e^{-2\alpha t}D(f\Vert \mathbb{E}f),
    \label{eq:circle-main-entropy-decay}
\end{equation} 
which is exponential decay of the matrix-valued entropy. 

In the following, we show that $\alpha_c(\Delta)=1$ for the circle. Indeed, we will show the exponential decay of matrix-valued Fisher information
\begin{equation}
    \mathcal I(H_tg)
      \leq e^{-2t}\mathcal I(g),
    \label{eq:circle-fisher-decay}
\end{equation}
which by integral over time yields the entropy decay \eqref{eq:circle-main-entropy-decay}.
The main point is the
following a matrix-valued version of
the Wirtinger inequality \cite{AmannEscher2009}.  Its proof uses the scalar bounds for the
logarithmic mean, a matrix logarithmic-center normalization, and
the Hilbert-valued Poincar\'e inequality.

\subsection{The logarithmic mean and the Fr\'echet logarithm}
\label{subsec:circle-logarithmic-mean}

For \(s,t>0\), let
\[
    \ell(s,t)=
    \begin{cases}
      \displaystyle\frac{s-t}{\log s-\log t},&s\neq t,\\[1.2ex]
      s,&s=t,
    \end{cases}
\]
be the logarithmic mean.  We shall repeatedly use
\begin{equation}
    \sqrt{st}\leq\ell(s,t)\leq\frac{s+t}{2}.
    \label{eq:log-mean-bounds}
\end{equation}
Indeed,
\[
    \ell(s,t)=\int_0^1s^u t^{1-u}\,du.
\]
The lower bound in \eqref{eq:log-mean-bounds} follows from Jensen's
inequality, and the upper bound follows by integrating the weighted
arithmetic--geometric mean inequality
\(s^ut^{1-u}\leq us+(1-u)t\).

Write
\[
    M_n^{++}
      =\{a\in(M_n)_{\mathrm{sa}}:
          a\geq\varepsilon1\text{ for some }\varepsilon>0\}.
\]
For \(a\in M_n^{++}\) and \(h=h^*\in M_n\), the
Fr\'echet derivative of the logarithm has the resolvent representation
\begin{equation}
    (D\log)_a(h):=\frac{d}{ds}\log(a+sh)|_{s=0}
      =\int_0^\infty(a+r1)^{-1}h(a+r1)^{-1}\,dr .
    \label{eq:frechet-log-resolvent}
\end{equation}
Assume the spectral decomposition of \(a\),
\[
    a=\sum_{i=1}^r\lambda_i p_i,
\]
 then Fr\'echet derivative also admits 
the double sum formula,
\begin{equation}
    (D\log)_a(h)
      =
      \sum_{i,j=1}^r
      \frac{1}{\ell(\lambda_i,\lambda_j)}\,p_i h p_j .
    \label{eq:frechet-log-spectral}
\end{equation}

\begin{lemma}[Two trace estimates]
\label{lem:frechet-log-trace}
Let \(a\in M_n^{++}\) and \(h=h^*\in M_n\).  Then
\begin{align}
    \operatorname{tr}\left(\bigl((D\log)_a(h)\bigr)^2\right)
      &\leq
      \operatorname{tr}(a^{-1}ha^{-1}h),
    \label{eq:frechet-log-upper}\\
    \operatorname{tr}\left((D\log)_a(h)a^{-1}h\right)
      &\geq
      \operatorname{tr}(a^{-1}ha^{-1}h).
    \label{eq:frechet-log-lower}
\end{align}
In particular, all quantities in
\eqref{eq:frechet-log-upper}--\eqref{eq:frechet-log-lower} are real.
\end{lemma}

\begin{proof}
Using the spectral decomposition above, put
\[
    c_{ij}=\operatorname{tr}(p_i h p_j h).
\]
Since \(h=h^*\), one has \(c_{ij}=c_{ji}\geq0\).  Formula
\eqref{eq:frechet-log-spectral} and trace cyclicity give
\begin{align*}
    \operatorname{tr}\left(\bigl((D\log)_a(h)\bigr)^2\right)
      &=
      \sum_{i,j=1}^r
      \frac{c_{ij}}{\ell(\lambda_i,\lambda_j)^2},\\
    \operatorname{tr}(a^{-1}ha^{-1}h)
      &=
      \sum_{i,j=1}^r\frac{c_{ij}}{\lambda_i\lambda_j}.
\end{align*}
The first inequality in \eqref{eq:log-mean-bounds} proves
\eqref{eq:frechet-log-upper}.

Similarly,
\[
    \operatorname{tr}\left((D\log)_a(h)a^{-1}h\right)
      =
      \sum_{i,j=1}^r
      \frac{c_{ij}}
           {\lambda_j\ell(\lambda_i,\lambda_j)}.
\]
By the symmetry of \(c_{ij}\), the last sum equals
\[
    \sum_{i,j=1}^r
      \frac{c_{ij}}{2\ell(\lambda_i,\lambda_j)}
      \left(\frac1{\lambda_i}+\frac1{\lambda_j}\right).
\]
The upper bound for the logarithmic mean gives
\[
    \frac{1}{2\ell(\lambda_i,\lambda_j)}
    \left(\frac1{\lambda_i}+\frac1{\lambda_j}\right)
      \geq\frac1{\lambda_i\lambda_j},
\]
which proves \eqref{eq:frechet-log-lower}.
\end{proof}

We also record a tracial differentiation identity that will be used
below.

\begin{lemma}[Trace identity]
\label{lem:frechet-log-trace-identity}
Let \(a\in M_n^{++}\), \(h=h^*\in M_n\), and let \(F\) be
a bounded Borel function on the spectrum set $\text{spec}(a)$ of $a$.  Then
\begin{equation}
    \operatorname{tr}\bigl(F(a)(D\log)_a(h)\bigr)
      =
    \operatorname{tr}\bigl(F(a)a^{-1}h\bigr).
    \label{eq:frechet-log-trace-identity}
\end{equation}
\end{lemma}

\begin{proof}
Using \eqref{eq:frechet-log-resolvent}, trace cyclicity, and the fact
that \(F(a)\) commutes with every resolvent of \(a\), we obtain
\begin{align*}
    \operatorname{tr}\bigl(F(a)(D\log)_a(h)\bigr)
      =
      \int_0^\infty
      \operatorname{tr}\bigl(
        F(a)(a+r1)^{-2}h
      \bigr)\,dr
      =
      \operatorname{tr}\bigl(F(a)a^{-1}h\bigr),
\end{align*}
where we use the elementary scalar identity
\(\int_0^\infty(x+r1)^{-2}\,dr=x^{-1}\) for $x>0$.
\end{proof}

\subsection{A nonlinear Wirtinger inequality}
\label{subsec:circle-nonlinear-wirtinger}

The normalization needed in the nonlinear argument is the matrix
analogue of subtracting the mean of a scalar function.

\begin{lemma}[Matrix logarithmic center]
\label{lem:circle-log-center}
Let \(a:\mathbb T\to M_n^{++}\) be norm continuous.  There is
strictly positive \(p\in M_n^{++}\) such that
\begin{equation}
    \int_{\mathbb T}
      \log\bigl(p^{-1/2}a(\theta)p^{-1/2}\bigr)
      \,d\mu(\theta)=0,
    \label{eq:circle-log-center}
\end{equation}
where the integral is a norm-Bochner integral.
\end{lemma}
\begin{proof}
Let \(\Omega=M_n^{++}\), equipped with the Thompson metric
\(d_T(x,y)=\|\log(x^{-1/2}yx^{-1/2})\|\), and let
\(\nu=a_*\mu\) be the push-forward of \(\mu\) under \(a\). Since
\(a\) is norm continuous and \(\mathbb T\) is compact, its image is
compact in \(\Omega\). In particular, there exist
\(0<\varepsilon\leq M<\infty\) such that
\(\varepsilon I_n\leq a(\theta)\leq MI_n\) for every
\(\theta\in\mathbb T\). Hence \(\nu\) is a compactly supported Borel
probability measure on \(\Omega\), and
\[
    \int_\Omega d_T(x,I_n)\,d\nu(x)
      =\int_{\mathbb T}\|\log a(\theta)\|\,d\mu(\theta)
      \leq \max\{|\log\varepsilon|,|\log M|\}<\infty.
\]
Thus \(\nu\in\mathcal P^1(\Omega,d_T)\).

By the Karcher-barycenter theorem
\cite[Theorem~7.4]{Lawson2020Karcher}, there exists a unique
\(p=\Lambda(\nu)\in\Omega\) satisfying
\(\int_\Omega\log(p^{-1/2}xp^{-1/2})\,d\nu(x)=0\). Since
\(\nu=a_*\mu\), the change-of-variables formula for push-forward
measures gives \eqref{eq:circle-log-center}. If the Karcher equation
is stated using the inverse convention, the same identity follows
from
\(\log(p^{1/2}x^{-1}p^{1/2})
=-\log(p^{-1/2}xp^{-1/2})\).

Finally, the integrand in \eqref{eq:circle-log-center} is norm
continuous on the compact space \(\mathbb T\), so the integral is a
norm-Bochner integral.
\end{proof}
For a smooth \(a:\mathbb T\to M_n^{++}\), write
\[
    b=p^{-1/2}ap^{-1/2},
    \qquad
    u_a=a^{-1}a',
    \qquad
    u_b=b^{-1}b',
\]
then
\begin{equation}
    u_b=p^{1/2}u_ap^{-1/2},
    \qquad
    u_b'=p^{1/2}u_a'p^{-1/2}.
    \label{eq:circle-congruence-log-derivative}
\end{equation}
Consequently,
\begin{equation}
    \operatorname{tr}(u_b^2)=\operatorname{tr}(u_a^2),
    \qquad
    \operatorname{tr}((u_b')^2)
      =\operatorname{tr}((u_a')^2).
    \label{eq:circle-congruence-energy}
\end{equation}

\begin{theorem}[Logarithmic-mean Wirtinger inequality]
\label{thm:circle-logarithmic-wirtinger}
Let \(a:\mathbb T\to M_n^{++}\) be a \(C^2\)-loop and put
\[
    u=a^{-1}a'.
\]
Then
\begin{equation}
    \int_{\mathbb T}\operatorname{tr}((u')^2)\,d\mu
      \geq
    \int_{\mathbb T}\operatorname{tr}(u^2)\,d\mu .
    \label{eq:circle-nonlinear-wirtinger}
\end{equation}
\end{theorem}

\begin{proof}
By Lemma~\ref{lem:circle-log-center} and
\eqref{eq:circle-congruence-energy}, we may replace \(a\) by a fixed
congruence and assume
\begin{equation}
    X=\log a,
    \qquad
    \int_{\mathbb T}X\,d\mu=0.
    \label{eq:circle-centered-log}
\end{equation}
Define
\[
    V=\int_{\mathbb T}\operatorname{tr}(X^2)\,d\mu,
    \qquad
    E=\int_{\mathbb T}\operatorname{tr}(u^2)\,d\mu,
    \qquad
    A=\int_{\mathbb T}\operatorname{tr}((u')^2)\,d\mu.
\]
The apparent non-self-adjointness of \(u\) or \(u'\) causes no positivity
problem.  Although $u$ and $u'$ are generally not self-adjoint, define
\begin{equation}\label{eq:q-similar-selfadjoint}
 \begin{aligned}
  Y&:=a^{-1/2}a'a^{-1/2}=Y^*,\\
  Z&:=a^{-1/2}(a''-a'a^{-1}a')a^{-1/2}=Z^*.
 \end{aligned}
\end{equation}
These are the self-adjoint representatives of $u$ and $u'$, respectively, under similarity.
Then one has \(u=a^{-1/2}Ya^{1/2}\) and \(u'=a^{-1/2}Za^{1/2}\), and hence
  \begin{align*}
      \operatorname{tr}(u^2)=\operatorname{tr}(Y^2)\geq0,\qquad
      \operatorname{tr}(u'^2)=\operatorname{tr}(Z^2)\geq0.
  \end{align*}

Equip the real vector space \((M_n)_{\mathrm{sa}}\) with the
Hilbert--Schmidt inner product. Consider the matrix-valued Fourier expansion,
\[ X(\theta)=\sum_{k\in \mathbb{Z}} \widehat{X}(k)e^{ik\theta},\quad  \widehat{X}(k)=\int_{\mathbb{T}} X(\theta)e^{-ik\theta}d\mu(\theta).\]
Since \(X\) has mean zero, the mean-zero condition gives
$\widehat X(0)=0$, and Parseval's identity gives
\[
 \int_{\mathbb{T}}\operatorname{tr}(X^2)\,d\mu
 =\sum_{k\ne0}\|\widehat X(k)\|_2^2,
 \qquad
 \int_{\mathbb{T}}\operatorname{tr}((X')^2)\,d\mu
 =\sum_{k\ne0}k^2\|\widehat X(k)\|_2^2,
\]
where $\|\cdot\|_2$ is the Hilbert-Schmidt 2 norm. Since $k^2\ge1$ for $k\ne0$, the gap-one Hilbert-valued Poincar\'e
inequality follows.  On the other hand, 
Lemma~\ref{lem:frechet-log-trace} gives
\begin{align}
    V=
      \int_{\mathbb T}\operatorname{tr}(X^2)\,d\mu
      &\leq
      \int_{\mathbb T}\operatorname{tr}((X')^2)\,d\mu \notag\\
      &=
      \int_{\mathbb T}
      \operatorname{tr}\left(
        \bigl((D\log)_a(a')\bigr)^2
      \right)\,d\mu\leq \int_{\mathbb T}\operatorname{tr}(a^{-1}a'a^{-1}a')\,d\mu
      = E.
    \label{eq:circle-V-less-E}
\end{align}
Define the function $r:\mathbb{T}\to \mathbb{R}$,
\[
    r(\theta)=\frac12\operatorname{tr}(X(\theta)^2).
\]
By Lemma~\ref{lem:frechet-log-trace-identity}, used with
\(F=\log\), we have
\begin{equation}
    r'
      =\operatorname{tr}(XX')
      =\operatorname{tr}(Xu).
    \label{eq:circle-r-prime}
\end{equation}
Differentiating once more and applying
\eqref{eq:frechet-log-lower} with \(h=a'\), we obtain
\begin{align}
    r''
      =
      \operatorname{tr}\bigl((D\log)_a(a')u\bigr)
      +\operatorname{tr}(Xu')
      \geq
      \operatorname{tr}(u^2)+\operatorname{tr}(Xu'),
    \label{eq:circle-r-second}
\end{align}
Notice that \(\int_{\mathbb{T}} r''\,d\mu=0\) since \(r\) is periodic, the  integration of
\eqref{eq:circle-r-second} gives
\begin{equation}
    E=\int_{\mathbb T}\operatorname{tr}(u^2)d\mu
      \leq
      -\int_{\mathbb T}\operatorname{tr}(Xu')\,d\mu.
    \label{eq:circle-E-cross}
\end{equation}

It remains to put the last expression into a manifestly Hilbertian
form. Because \(X=\log a\) commutes with every Borel function of \(a\),
trace cyclicity yields
\begin{equation}
    \operatorname{tr}((u')^2)=\operatorname{tr}(Z^2),
    \qquad
    \operatorname{tr}(Xu')=\operatorname{tr}(XZ).
    \label{eq:circle-similarity-CS}
\end{equation}
In particular, \(A=\int_{\mathbb T}\operatorname{tr}((u')^2)d\mu \geq0\).  The tracial Cauchy--Schwarz
inequality, \eqref{eq:circle-E-cross}, and
\eqref{eq:circle-V-less-E} now imply
\[
    E
      \leq
      \left|
        \int_{\mathbb T}\operatorname{tr}(XZ)\,d\mu
      \right|
      \leq\sqrt{VA}
      \leq\sqrt{EA}.
\]
If \(E>0\), squaring and dividing by \(E\) gives
\(A\geq E\).  If \(E=0\), the same
conclusion follows from \(A\geq0\).
\end{proof}

Lawson's theorem the above is used only to obtain a solution of the centering
equation.  After this normalization, the proof is entirely analytic.

\subsection{Resolvent dissipation and Fisher-information decay}
\label{subsec:circle-fisher-decay}

Let \(g:\mathbb T\to M_n^{++}\) be smooth and set
\(g_t=H_tg\).  For \(s\geq0\), define
\[
    k_{t,s}=g_t+s1,
    \qquad
    u_{t,s}=k_{t,s}^{-1}k'_{t,s},
\]
and introduce the resolvent energy
\begin{equation}
    \mathcal E_s(t)
      =
      \int_{\mathbb T}\operatorname{tr}(u_{t,s}^2)\,d\mu.
    \label{eq:circle-resolvent-energy}
\end{equation}
As in the proof of Theorem~\ref{thm:circle-logarithmic-wirtinger},
\(u_{t,s}\) is similar to a self-adjoint element, so
\(\mathcal E_s(t)\geq0\). In the following, if no confusion, we will often use $k=k_{t,s}, u=u_{t,s}$ for the ease of the notations.

\begin{proposition}[Resolvent-energy dissipation]
\label{prop:circle-resolvent-dissipation}
For every \(s\geq0\),
\begin{equation}
    \frac{d}{dt}\mathcal E_s(t)
      =
      -2\int_{\mathbb T}
        \operatorname{tr}\bigl((u_{t,s}')^2\bigr)\,d\mu
      \leq-2\mathcal E_s(t).
    \label{eq:circle-resolvent-dissipation}
\end{equation}
Consequently,
\begin{equation}
    \mathcal E_s(t)\leq e^{-2t}\mathcal E_s(0).
    \label{eq:circle-resolvent-decay}
\end{equation}
\end{proposition}

\begin{proof}
Put \(v=k^{-1}k''\), use $\partial_t$ for the time derivative, and retain the
prime $'$ for the circle derivative with respect to $\theta$.  The heat equation \(\partial_t k=k''\)
gives
\[
    u'=v-u^2,
    \qquad
    \partial_t u=v'+[u,v].
\]
\begin{align*}
    \partial_t u
      &=-k^{-1}(\partial_t k)\,k^{-1}k'+k^{-1}(\partial_t k')
        =-k^{-1}k''k^{-1}k'+k^{-1}k''',\\
    v'
      &=-k^{-1}k'k^{-1}k''+k^{-1}k''',\\
    \partial_t u-v'
      &=-vu+uv=[u,v].
\end{align*}
Trace cyclicity eliminates the commutator term 
\[\operatorname{tr}(u[u,v])=\operatorname{tr}(u^2v-uvu)=0.\]  Periodic integration
by parts then gives
\begin{align*}
    \frac12\frac{d}{dt}\mathcal E_s(t)
      &=
      \int_{\mathbb T}\operatorname{tr}(u(\partial_t u))\,d\mu
      \quad\text{(substitute $\partial_t u = v' + [u,v]$)}\\
      &=
      \int_{\mathbb T}\operatorname{tr}(uv')\,d\mu
      \quad\text{(integration by parts)}\\
      &=
      -\int_{\mathbb T}\operatorname{tr}(u'v)\,d\mu
      \quad\text{(substitute $v = u' + u^2$)}\\
      &=
      -\int_{\mathbb T}\operatorname{tr}((u')^2)\,d\mu
      -\int_{\mathbb T}\operatorname{tr}(u'u^2)\,d\mu.
      \quad\text{ }
\end{align*}
The last integral vanishes because
\[
    \int_{\mathbb T}\operatorname{tr}(u'u^2)\,d\mu
      =
      \frac13\int_{\mathbb T}
      \bigl(\operatorname{tr}(u^3)\bigr)'\,d\mu=0.
\]
This proves the identity in
\eqref{eq:circle-resolvent-dissipation}.  Its inequality follows by
applying Theorem~\ref{thm:circle-logarithmic-wirtinger} to
\(k_{t,s}\), and \eqref{eq:circle-resolvent-decay} follows from
Gronwall's lemma.
\end{proof}

The resolvent-energy dissipation immediately implies
the decay of
Bogoliubov--Kubo--Mori Fisher information by integral. Indeed, integration by parts and
\eqref{eq:frechet-log-resolvent} yield
\begin{align}
    \mathcal I(g_t)
      &=
      \int_{\mathbb T}
      \operatorname{tr}\bigl(
        g_t'(D\log)_{g_t}(g_t')
      \bigr)\,d\mu
      \notag\\
      &=
      \int_0^\infty\int_{\mathbb T}
      \operatorname{tr}\bigl(
        (g_t+s1)^{-1}g_t'
        (g_t+s1)^{-1}g_t'
      \bigr)\,d\mu\,ds
      =\int_0^\infty\mathcal E_s(t)\,ds.
    \label{eq:circle-fisher-resolvent}
\end{align}
As the scalar integrands in
\eqref{eq:circle-fisher-resolvent} are nonnegative,  Tonelli's theorem
and \eqref{eq:circle-resolvent-decay} therefore give
\[
    \mathcal I(H_tg)
      \leq
      e^{-2t}\mathcal I(g),
\]
which yields the exponential entropy decay \eqref{eq:circle-main-entropy-decay} for smooth $g$. The extension to general positive Borel matrix valued function $f\in L_1(\mathbb{T},M_n^+)$ can be obtained by a standard approximation argument (see \cite[Appendix]{BrannanGaoJunge2022}).

\begin{remark}[Finite tracial coefficients]{\rm 
\label{rem:finite-tracial-extension} The sharp modified log-Sobolev inequality obtained above remain valid with
\(M_n\) replaced by an arbitrary finite tracial von Neumann algebra, which is needed for the transference to Quantum Tori in Corollary \ref{cor:sharp-tori}.
Indeed, in Lemma~\ref{lem:frechet-log-trace} the finite spectral sum can be the double-operator-integral formula
for the divided difference with the logarithm
(see e.g. \cite[Equation~(1.1)]{Peller2006MOI}).  The logarithmic center remains
available by Lawson's Karcher-mean theorem for unital
\(C^*\)-algebras \cite[Theorem~7.4]{Lawson2020Karcher}.  All remaining
steps use only traciality and the Hilbert-space Poincar\'e inequality.}
\end{remark}

\bibliographystyle{amsalpha}
\bibliography{section2_references}

\end{document}